\documentclass{article}
\usepackage[utf8]{inputenc}
\usepackage{mathtools}
\usepackage{amstext}
\usepackage{amsfonts}
\usepackage{amssymb}
\usepackage{amsthm}
\usepackage{amsmath}
\usepackage{url}
\usepackage{cite}
\title{Note on the union-closed sets conjecture}
\author{Abigail Raz\thanks{Department of Mathematics, Rutgers University, Piscataway NJ. Email: ajr224@math.rutgers.edu}}
\date{}
\def\A{{\mathcal A}}
\def\F{{\mathcal F}}
\newtheorem{theorem}{Theorem}
\newtheorem{counter}{Counterexample}
\newtheorem{note}{Note}
\newtheorem{fact}{Fact}
\newtheorem{conj}{Conjecture}
\newtheorem{defn}{Definition}

\begin{document}

\maketitle

\begin{abstract}
The union-closed sets conjecture states that if a family of sets $\A \neq \{\emptyset\}$ is union-closed, then there is an element which belongs to at least half the sets in $\A$. In 2001, D. Reimer showed that the average set size of a union-closed family, $\A$, is at least $\frac{1}{2} \log_2 |\A|$. In order to do so, he showed that all union-closed families satisfy a particular condition, which in turn implies the preceding bound. Here, answering a question raised in the context of T. Gowers' polymath project on the union-closed sets conjecture, we show that Reimer's condition alone is not enough to imply that there is an element in at least half the sets.
\end{abstract}

\section{Introduction}
% Say union-closed conj
%Give Reimer's result
% Introduce notation [n], [A,B] and such
% Give our question and say motivated by Gowers
Given the set $[n]= \{1, \ldots, n\}$ and a family $\A \subseteq 2^{[n]}$ we say $\A$ is union-closed if for $A, B \in \A$ we have $A \cup B \in \A$.  The Union-Closed Sets Conjecture, due to P. Frankl \cite{Rival}, states that if $\A \subseteq 2^{[n]}$ is union-closed and $\A \neq \{\emptyset\}$ then there is some element of $[n]$ which belongs to at least half the sets in $\A$.  One method of approaching this conjecture is to look at the average frequency of an element or, equivalently, the average set size. The following theorem of D. Reimer \cite{Reimer} was thus motivated by and can be shown to follow from, the union-closed sets conjecture.

\begin{theorem}\label{Reimer} 
If $\A \subseteq 2^{[n]}$ and is union-closed, then

\begin{equation}\frac{\sum_{A \in \A} |A|}{|\A|} \ge \frac{\log_2|\A|}{2}\end{equation}

\end{theorem}

We will say that $\F \subseteq 2^{[n]}$ is a filter if $G \supseteq F$ and $F \in \F$ implies $G \in \F$. Additionally, for $A \subseteq B \subseteq [n]$ define $[A,B]\coloneqq \{C : A \subseteq C \subseteq B\}$.
In order to prove Theorem \ref{Reimer}, Reimer introduced the following criterion for a family $\A \subseteq 2^{[n]}$:\\
\begin{defn}\label{def}
We say $\A \subseteq 2^{[n]}$ satisfies \emph{Condition 1}
if there exists a filter $\mathcal{F} \subseteq 2^{[n]}$ and a bijection $A \mapsto F_A$ from $\A$ to $\F$ satisfying:
\begin{enumerate}
    \item $A \subseteq F_A$ for all $A \in \A$
    \item For distinct $A, B \in \A$ we have $[A, F_A] \cap [B, F_B] = \emptyset$.
\end{enumerate}
\end{defn}

Reimer's proof of Theorem \ref{Reimer} consists of two steps.  He first shows that every union-closed family $\A$ satisfies Condition 1.
He then shows that Condition 1 implies Theorem \ref{Reimer}.

In 2016, T. Gowers began a polymath project focused on the union-closed sets conjecture. In the comments on the initial post I. Balla first proposed the conjecture below. Gowers reiterates this conjecture in his second post focused on strengthenings of the union-closed sets conjecture. In the comments there is a discussion of a possible counterexample, and it is stated that all families with ground set at most 5 and a random sampling of families with ground set at most 12 have been confirmed to satisfy the conjecture \cite{Gowers}. 
%This problem was suggested to me by J. Kahn.
\begin{conj}\label{conj}
Assume $\A\subseteq 2^{[n]}$ satisfies Condition 1. Then there is an element $x\in [n]$ in at least half the sets of $\A$.
\end{conj}

As Reimer showed that all union-closed families satisfy Condition 1, this conjecture is clearly a strengthening of the union-closed sets conjecture. The purpose of this note is to show that Conjecture \ref{conj} is false.

%\begin{note}
%An equivalent way of stating condition 2 is that either there is $a \in A$ such that $a \notin F_B$ or there is $b \in B$ such that $b \notin F_A$.
%\end{note}

\section{Counterexample}
In what follows we will always have $\A$ and $ \F$ as in Definition \ref{def}. 

\begin{note} \label{equiv}
An equivalent way of stating the second part of Condition 1 is that at least one of $A \setminus F_B$ or $B \setminus F_A$ is non-empty.
%there is either $a \in A$ such that $a \notin F_B$ or there is $b \in B$ such that $b \notin F_A$.
\end{note}

We will use the following notation:
\begin{itemize}
    \item $\A_x = \{A \in \A : x \in A\}$
    \item $A_0$ is the set for which $F_{A_0}=[n]$
    \item $A_i$ is the set for which $F_{A_i} = [n]\setminus \{i\}$ for $i \in [n]$
    \item $B_{i,j}$ is the set for which $F_{B_{i,j}}= [n]\setminus \{i,j\}$ for $i \neq j \in [n]$
\end{itemize}

Before giving the counterexample we will briefly describe how we found it and indicate why no smaller example is possible. The following observation was our starting point. 
\begin{fact} \label{n-1}
Assume $\A$ satisfies Condition 1. If every set in $\F$ has size at least $n-1$ then there is an element in at least half of the sets of $\A$.
\end{fact}

\begin{proof}
%First we may assume that for every $x \in [n]$ there is an $A \in \A$ such that $x \in A$. 
Without loss of generality assume $\F= \{[n]\} \cup \{[n]\setminus \{i\}: i \in [k]\}$. Hence, $|\F|=|\A| = k+1$. By Note \ref{equiv} we know that $[k] \subseteq A_0$. Now we will view each $A_i$ as a vertex labelled $i$ in a digraph, $D$, on vertex set $[k]$, with $(i,j)$ an edge exactly when $i \in A_j$. Again by Note \ref{equiv} we know that $D$ must contain a tournament. Furthermore, the number of sets containing $i$ is simply the out-degree of $i$ plus 1 (since $i \in A_0$). Since $D$ has $k$ vertices and contains a tournament it has maximum out-degree at least $\frac{k-1}{2}$. Hence there is always an element in at least $\frac{k+1}{2}$ members of $\A$.
\end{proof}

We first observe that if $n$ is the smallest integer such that there is a counterexample to Conjecture \ref{conj} on $[n]$ and $\A$ is such a counterexample, then $\F$ must contain all sets of size $n-1$. To see this suppose instead that the elements of $\F$ of size $n-1$ are $[n]\setminus \{i\}$ for $i \in [k]$ with $k<n$. Since $\F$ is a filter we have $\{k+1, \ldots, n\} \subseteq F$ for all $F \in \F$, implying that the condition in Note \ref{equiv} is not affected if we replace each $X \in \A \cup \F$ by $X\setminus \{k+1, \ldots,n\}$. This produces a counterexample on a smaller set, contradicting the minimality of $n$. 

%Let $\F^* = \{F\setminus \{k+1, \ldots, n\}: F\in \F\}$ and $\A^* = \{A \setminus \{k+1, \ldots, n\} : A \in \A\}$. $\A$ is a counterexample, so by Note \ref{equiv} we cannot have $A,B \in \A$ such that $A \cup \{k+1, \ldots, n\}= B$. Furthermore, for $A, B \in \A$ we have $A \setminus F_B \subseteq [k]$. Thus $\A^*$ is a counterexample of the same size as $\A$ and $\F^*$ is of the desired form. 

Restrict $\A$ to $\A'\coloneqq\{A_i\}_{i=0}^n$. If $n$ is even then there exists $x\in [n]$ with $|\A'_x| \ge \frac{n+2}{2}$.  Hence we need at least two sets in $\A \setminus \A'$. (If $n$ is odd similar reasoning shows that there must be at least three sets in $\A \setminus \A'$.) 

In our example we will take $n$ to be even and $\F$ to consist of $[n] \setminus \{1,2\}$ and $ [n] \setminus\{3,4\}$ along with all sets of size at least $n-1$. Thus $|\F|=|\A|=n+3$, $A_0= [n]$, and we want to arrange that $|\A_x| \le \frac{n}{2}+1$ for all $x \in [n]$. We will use the same digraph, $D$, as in the proof of Fact \ref{n-1} (with $(i,j)$ an edge if and only if $i \in A_j$). Note that the $B_{i,j}$'s do not affect the digraph. By Note \ref{equiv} we know that if $B_{i,j} \in \A$ then $i \in A_j$ and $j \in A_i$. Therefore, the sum of the out-degrees in $D$ must be at least $\frac{n^2-n}{2}+2$. Without loss of generality $1 \in B_{3,4}$, since $B_{1,2}$ and $B_{3,4}$ must satisfy the condition of Note \ref{equiv}. Additionally, if $B_{1,2}= \emptyset$ then to satisfy Note \ref{equiv} all other sets in $\A$ must contain 1 or 2.  However, $A_0$ contains both 1 and 2, so one of 1 or 2 must appear in at least half the sets, contradicting that $\A$ is a counterexample. Hence, $B_{1,2}$ and $B_{3,4}$ are both non-empty, so we must have at least $2$ vertices of out-degree no more than $\frac{n}{2}-1$ and the rest of out-degree no more than $\frac{n}{2}$. (If $|\A \setminus \A'| >2$ then we get even more ``extra" degrees and the following lower bound on $n$  increases.) Thus we have the inequality $2(\frac{n}{2}-1)+(n-2)(\frac{n}{2}) \ge \frac{n^2-n}{2}+2$, i.e. $n \ge 8$. When $n$ is odd similar consideration gives $n \ge 13$; so, since our example does indeed use $n=8$ it is of the smallest possible size. 

%Note that if we used a similar argument for odd $n$ we find that $n \ge 13$, implying that our counterexample is of the smallest possible size. 
\begin{counter}
Here we will take our universe to be $[8]$. Our family $\A$ consists of the following 11 sets:
\begin{itemize}
    \item $A_0 = [8]$
    \item $A_1= \{2, 4, 6, 7, 8\}$
    \item $A_2= \{1,3,5,8\}$
    \item $A_3= \{1,4,7,8\}$
    \item $A_4 = \{2,3,5,6\}$
    \item $A_5 = \{1, 3, 7\}$
    \item $A_6 = \{2, 3,5\}$
    \item $A_7 = \{2, 4, 6\}$
    \item $A_8 = \{4, 5, 6, 7\}$
    \item $B_{1,2} = \{8\}$
    \item $B_{3,4} = \{1\}$
    
\end{itemize}
%The corresponding filter $\F$ has $F_{A_0}= [19]$, $F_{A_i} = [19] \setminus \{i\}$, and $F_{B_{i,j}} = [19] \setminus \{i,j\}$.
\end{counter}

We (or our computers) can easily check that the requirement in Note \ref{equiv} is satisfied and that each element appears in exactly 5 sets.
\\
\\
{\bf Acknowledgment:} I would like to thank Jeff Kahn for suggesting this problem.

\bibliographystyle{amsplain}
\bibliography{Counterexampledesk}

\providecommand{\bysame}{\leavevmode\hbox to3em{\hrulefill}\thinspace}
\providecommand{\MR}{\relax\ifhmode\unskip\space\fi MR }
\providecommand{\MRhref}[2]{%
  \href{http://www.ams.org/mathscinet-getitem?mr=#1}{#2}
}
\providecommand{\href}[2]{#2}
\begin{thebibliography}{1}

\bibitem{Gowers}
Timothy Gowers, \emph{Func1 --- strengthenings, variants, potential
  counterexamples},
  \url{gowers.wordpress.com/2016/01/29/func1-strengthenings-variants-potential-counterexamples},
  2006.

\bibitem{Reimer}
David Reimer, \emph{An average set size theorem}, Combinatorics, Probability,
  and Computing \textbf{12} (2003), no.~1, 89--93.

\bibitem{Rival}
Ivan Rival (ed.), \emph{Graphs and order}, 1 ed., Nato Science Series C:, vol.
  147, Springer, 1985.

\end{thebibliography}

\end{document}